\documentclass[a4paper,10pt]{article}

\usepackage{multido}
\usepackage{amssymb,amsmath,amsthm}

\theoremstyle{plain}
\newtheorem{Def}{Definition}[section]
\newtheorem{Theor}[Def]{Theorem}
\newtheorem{Corol}[Def]{Corollary}
\newtheorem{Prop}[Def]{Proposition}
\newtheorem{Lem}[Def]{Lemma}

\theoremstyle{remark}
\newtheorem{Rem}[Def]{Remark}

\newcommand{\cR}{{\mathbb R}}
\newcommand{\non}{\nonumber}
\newcommand{\eq}[1]{\mbox{\rm {(\ref{#1})}}}
\newcommand{\ve}{\varepsilon}
\newcommand{\na}{\nabla}

\newcommand{\Om}{\Omega}
\newcommand{\Sig}{\Sigma^{\rm lin}}
\newcommand{\Sigbar}{\bar\Sigma^{\rm lin}}
\newcommand{\si}{\hat\sigma}

\DeclareMathOperator{\di}{div}
\DeclareMathOperator{\Curl}{Curl}
\DeclareMathOperator{\dom}{dom}

\DeclareMathOperator{\inter}{int}
\DeclareMathOperator{\rt}{rt}

\newcommand{\yieldlimit}{\sigma_{\mathrm{y}}}

\newcommand{\R}{\mathbb{R}}

\newcommand{\C}{\mathbb{C}}

\DeclareMathOperator{\sym}{sym}
\DeclareMathOperator{\Tr}{tr}

\DeclareMathOperator{\dev}{dev}

\DeclareMathOperator{\so}{\mathfrak{so}}

\begin{document}
\title{Well-posedness for dislocation based gradient visco-plasticity with isotropic hardening}
\author{N. Kraynyukova%
\thanks{Nataliya Kraynyukova, Fachbereich Mathematik, 
Technische Universit\"at Darmstadt, 
Schlossgartenstrasse 7, 64289 Darmstadt, Germany, email: kraynyukova@mathematik.tu-darmstadt.de, 
Tel.: +49-6151-16-3287},\,
P. Neff%
\thanks{Patrizio Neff, Lehrstuhl f\"ur Nichtlineare Analysis und Modellierung, Fakult\"at f\"ur  Mathematik, Universit\"at Duisburg-Essen, Campus Essen, Thea-Leymann-Stra§e 9, 45127 Essen, Germany, email: patrizio.neff@uni-due.de, Tel.: +49-201-183-4243},\,
S. Nesenenko%
\thanks{Corresponding author: Sergiy Nesenenko, Fakult\"at f\"ur  Mathematik, Universit\"at Duisburg-Essen, Campus Essen, Thea-Leymann-Stra§e 9, 45127 Essen, Germany, email: sergiy.nesenenko@uni-due.de},\,
K. Che{\l}mi\'nski%
\thanks{Krzysztof Che{\l}mi\'nski, Faculty of Mathematics and Information Science, Warsaw University of Technology, Koszykowa 75, 00-662 Warsaw, Poland}
}
\maketitle

\begin{abstract}
In this work we establish the well-posedness for infinitesimal dislocation based
gradient viscoplasticity with isotropic hardening for general gradient monotone plastic flows. We assume an additive split of the displacement gradient into non-symmetric elastic distortion and non-symmetric plastic distortion. The thermodynamic potential is augmented with a term taking the dislocation density tensor $\Curl p$ into account. 
The constitutive equations in the models we 
study are assumed to be of self-controlling type. Based on the generalized version of Korn's inequality for incompatible tensor fields (the non-symmetric plastic distortion)
due to Neff/Pauly/Witsch the existence of solutions of quasi-static initial-boundary value problems 
under consideration is shown using a time-discretization technique and a monotone operator method.  
\end{abstract} 

\\[2ex]
\section{Introduction}
Within the framework of the strain gradient plasticity theory
we study the existence of solutions of
quasistatic initial-boundary value problems arising in gradient viscoplasticity with isotropic hardening. 
The models we study use rate-dependent constitutive equations with internal variables to describe the
deformation behaviour of metals at infinitesimally small strain. 

In this paper we consider the rate-dependent (viscoplastic) behaviour only. 
The model we study here has been first presented in \cite{Neff_Chelminski07_disloc} and 
has been inspired by the early work of Menzel and Steinmann \cite{Steinmann00}. Contrary to more classical strain gradient approaches, the model features a non-symmetric plastic distortion field $p\in{\cal M}^3$ 
(see \cite{Bardella10}),
a dislocation based energy storage based solely on $|\Curl p|$ and second gradients of the plastic 
distortion in the form of $\Curl\Curl p$ acting as dislocation based kinematical backstresses. 
Preliminary works on the problem concern the uniqueness (see \cite{Neff_Iutam08}),
 the well-posedness of a rate-independent variant
(see \cite{Ebobisse_et_2014,Ebobisse_Neff09,Neff_Chelminski07_disloc}) and of a rate-dependent variant without isotropic hardening (see \cite{NesenenkoNeff2011,NesenenkoNeff2013}), 
the possibility of homogenization 
(see \cite{NesenenkoNeff2014}), as well as FEM-implementations 
(see \cite{Neff_Sydow_Wieners08,Reddy06}). In \cite{Lussardi08,Reddy06} the well-posedness for a rate-independent model of Gurtin and Anand \cite{Gurtin05b} is shown under the decisive assumption that the plastic distortion is symmetric (the irrotational case), in which case we may really speak of a strain gradient plasticity model, since the gradient acts on the plastic strain. 
\paragraph{Presentation of the model.}
For completeness of the work we sketch some of the ingredients of the model.
The sets, ${\cal M}^3$ and ${\cal S}^3$ denote the sets of all
$3 \times 3$--matrices and of all symmetric $3 \times 3$--matrices,
 respectively. We recall that the space of all 
$3 \times 3$--matrices  ${\cal M}^3$ can be isomorphically identified 
 with the space ${\mathbb R}^9$. Therefore we 
can define a linear mapping $B: {\mathbb R}^N\to{\cal M}^3$ as a composition
of a projector from ${\mathbb R}^N$ onto ${\mathbb R}^9$ and the isomorphism
between ${\mathbb R}^9$ and ${\cal M}^3$. The transpose 
$B^T: {\cal M}^3 \rightarrow {\mathbb R}^N$ is given then by
$$B^T{\tau}=(\hat{z},0)^T$$
for $\tau\in{\cal M}^3$ and $z=(\hat{z}, \tilde{z})^T\in{\mathbb R}^N,$ 
$\tilde{z}\in{\mathbb R}^{N-9}$.
 Next, as is usual in plasticity theory, we split the total displacement gradient into non symmetric elastic and plastic distortions
\begin{align*}
  \nabla u=e+p\, .
\end{align*}
For invariance reasons, the elastic energy contribution may only depend on the elastic strains $\sym e=\sym (\nabla u-p)$. 
While $p$ is non-symmetric, 
a distinguishing feature of the model is that, similar to classical approaches, 
only the symmetric part $\varepsilon_p:=\sym p$ of the plastic distortion appears in the local Cauchy stress $\sigma$, while the higher order stresses are non-symmetric (see \cite{Neff_techmech07,Neff_Svendsen08} for more details). We assume as well plastic incompressibility $\Tr{p}=0$.
 We consider here a free
energy of the form
\begin{eqnarray}\label{free-eng}
\Psi(\nabla u,\Curl p,z):
&=&\underbrace{\Psi^{\mbox{\scriptsize lin}}_e(e)}_{\mbox{\small
elastic energy}}\,\, +\,\,\,\underbrace{\Psi^{\mbox{\scriptsize
lin}}_{\Curl}(\Curl p)}_{\mbox{\small defect
energy (GND)}}\\
\nonumber &+&\underbrace{\Psi_{\mbox{\scriptsize
hard}}(z)}_{\mbox{\small hardening energy (SSD)}}
 \end{eqnarray} where
\[\Psi^{\rm{\scriptsize lin}}_e(e):=\frac12 e\cdot \C e,\quad\Psi^{\rm{\scriptsize
lin}}_{\Curl}(\Curl p)=\frac{C_1}2\|\Curl p\|^2\,{\rm and }\,
\Psi_{\rm\scriptsize hard}(z)=\frac12 Lz\cdot z\,.\]
Here, the linear mapping
$L:{\mathbb R}^N\to{\mathbb R}^N$ corresponds to isotropic hardening effects and is assumed to be positive semi-definite and $C_1$ is a given non-negative material constant.
The positive definite elasticity tensor $\mathbb{C}$ is able to represent 
the elastic anisotropy of the material. The local free-energy imbalance states that
\begin{equation}
\dot{\Psi} - \sigma\cdot \dot{e} - \sigma\cdot \dot{p}  \leq 0\ .
\label{2ndlaw}
\end{equation}
Now we expand the first term, substitute (\ref{free-eng}) and get
\begin{equation}\label{exp-2ndlaw}
(\mathbb{C}e-\sigma)\cdot \dot{e}-\sigma\cdot \dot{p}+
C_1\Curl p\cdot \Curl\dot{p}+Lz\cdot \dot{z}\leq0\,
,
\end{equation}
which, using arguments from thermodynamics gives the elastic
relation
\begin{equation}
\sigma = {\C} \sym(\nabla u-p)
\label{elasticlaw}
\end{equation}
and the reduced dissipation inequality
\begin{equation}\label{reduced-diss}-\sigma\cdot \dot{p}+
C_1\Curl p\cdot \Curl\dot{p}+Lz\cdot \dot{z}\leq0\,.
\end{equation} 
Now we integrate
(\ref{reduced-diss}) over $\Omega$ and  get
\begin{eqnarray}
\nonumber0&\geq&\int_{\Omega}\Bigl[-\sigma\cdot \dot{p}+
C_1\Curl p\cdot \Curl\dot{p}+Lz\cdot \dot{z}\Bigr]\\
\nonumber&=&-\int_{\Omega}\Bigl[\sigma\cdot\dot{p}
-C_1\Curl\,\Curl p\cdot \dot{p}-Lz\cdot\dot{z}\\
&& \qquad +\sum_{i=1}^3\int_{\partial\Omega}
C_1\langle\dot{p}^i\times(\Curl p)^i,\vec{n}\rangle dS
\,.\label{reduced-diss2}
\end{eqnarray}
In order to obtain a dissipation inequality in the spirit of
classical plasticity, we assume that the infinitesimal plastic
distortion $p$ satisfies the so-called {\it linearized insulation
condition} 
\begin{equation}
\sum_{i=1}^3\int_{\partial\Omega}
C_1\,\langle\frac{d}{dt}p^i\times(\Curl p)^i,\vec{n}\rangle
dS=0.\label{lin-sul}
\end{equation} 
We specify a sufficient condition for the linearized insulation
 boundary condition (see \cite{Gurtin05}), namely \begin{equation}\label{bc-plastic}
      p\times n|_{\partial\Omega}=0,
\end{equation}
which is called the micro-hard boundary condition.
Under (\ref{bc-plastic}), we then obtain the dissipation inequality
\begin{equation}
\int_{\Omega} [(B^T\sigma +\Sigma^{\mbox{\scriptsize
lin}}_{\Curl})\cdot \dot {z} +\hat g\cdot\dot{z}]dV\geq0
\, ,\label{diss-ineq}\end{equation}
where $$\Sigma^{\mbox{\scriptsize lin}}_{\Curl}
:=-C_1B^T\,\Curl\,\Curl p \quad\mbox{ and }\quad
\hat g:=-Lz\, .$$
Adapted to our situation, the plastic flow has the form
\begin{align}
  \partial_t z\in g(B^T\sigma-Lz-C_1B^T\Curl\Curl p)\, ,
\end{align}
where $g$ is a multivalued monotone flow function which is not necessary the subdifferential of a convex plastic potential (associative plasticity). 

We note that the micro-hard boundary (\ref{bc-plastic})  is the correct boundary condition 
for tensor fields in $L^2_{\Curl}(\Omega, {\cal M}^3)$ (see Notation for the definition of $L^2_{\Curl}$-space) which admits tangential traces.
We combine this with a new inequality extending Korn's inequality 
to incompatible tensor fields, namely
\begin{align}
\label{incompatible_korn}
\forall \, p\in L^2_{\Curl}(\Omega, {\cal M}^3): \quad & p\times n|_{\partial\Omega}=0:   \\
  & \underbrace{\|p\|_{L^2(\Omega)}}_{\text{plastic distortion}}\le C(\Omega)\, 
     \Big( \underbrace{\|\sym p\|_{L^2(\Omega)}}_{\text{plastic strain}}+ \underbrace{ \|\Curl p\|_{L^2(\Omega)}}_{\text{dislocation density}} \Big)\, .\notag
\end{align}
Here, the domain $\Omega$ needs to be {\bf sliceable}, i.e. cuttable into finitely many simply connected subdomains with Lipschitz boundaries. This inequality has been derived in \cite{Neff_Pauly_Witsch_cracad11, Neff_Pauly_Witsch_Korn_diff_forms_m2as12,Neff_Pauly_Witsch_Sbornik12
} and is precisely motivated by the well-posedness question for our model 
\cite{Neff_Chelminski07_disloc}. The inequality \eqref{incompatible_korn} 
expresses the fact that controlling the plastic strain $\sym p$ and the dislocation 
density $\Curl p$ in $L^2(\Omega, {\cal M}^3)$ gives a control of the plastic distortion $p$ in 
$L^2(\Omega, {\cal M}^3)$ provided the micro-hard boundary condition is specified. 

It is worthy to note that with $g$ only monotone and not necessarily a subdifferential the powerful energetic solution concept \cite{Lussardi08,Kratochvil10,Mielke09} cannot be applied. In this contribution we face the combined challenge of a gradient plasticity model based on the dislocation density tensor $\Curl p$ involving the plastic spin, a general non-associative monotone flow-rule and a rate-dependent response. 
\paragraph{Setting of the problem.}
Let $\Omega \subset \cR^3$ be an open bounded set, the set of material
points of the solid body, with a $C^1$-boundary $\partial\Omega$. By $T_e$ we denote a positive number (time of existence),
which can be chosen arbitrarily large,  and for $0 < t\leq T_e$
\begin{eqnarray}
{\Omega}_{t} = {\Omega} \times {(0, t)}. \non
\end{eqnarray} 
 Let $\mathfrak{sl}(3)$ be the set of all traceless $3 \times 3$--matrices,
 i.e. $$\mathfrak{sl}(3)=\{v\in{\cal M}^3\mid \Tr v=0\}.$$
 Unknown in our small strain formulation are the displacement $u(x,t) \in
\cR^3$ of the material point $x$ at time $t$ and the vector of the internal variables 
$z=(p,\gamma)$. Here, $p(x,t)\in\mathfrak{sl}(3)$ denotes the non-symmetric infinitesimal
plastic distortion and $\gamma(x,t)\in{\mathbb R}$ is the isotropic hardening variable. 
The model equations of the problem are
\begin{eqnarray}
- \di_x \sigma(x,t) &=&  b(x,t), \label{CurlPr1}
\\[1ex]
\sigma(x,t) &=& {\mathbb C}[x] ( \sym (\na_x u(x,t) - Bz(x,t) ) ),   \label{CurlPr2}
\\[1ex] 
\label{CurlPr3} \partial_t z(x,t) & \in & g \big(x,\Sigma^{\rm lin}(x,t)\big),  
\hspace{3ex} \Sigma^{\rm lin}=\Sigma^{\rm lin}_e+\Sigma^{\rm lin}_{\rm sh}
+\Sigma^{\rm lin}_{\rm curl},\label{microPr3} \\
\Sigma^{\rm lin}_{\rm e}&=&B^T\sigma,\hspace{1ex} \Sigma^{\rm lin}_{\rm sh}=-Lz,
\hspace{1ex} \quad \Sigma^{\rm lin}_{\rm curl}=- C_1B^T \Curl\Curl (Bz)\, ,\non
\end{eqnarray}
which must be satisfied in $\Omega \times [0,T_e)$. Here, $\Sigma^{\rm lin}$ is the infinitesimal Eshelby stress tensor driving the evolution of the plastic distortion $p$. The initial condition and Dirichlet boundary condition are  
\begin{eqnarray} 
z(x,0) &=& 0, \quad\quad\quad x \in \Omega,  \label{CurlPr4}   \\
Bz(x,t)\times n(x)&=&0,\hspace{7ex} (x,t) \in \partial\Omega \times
  [0,T_e), \label{CurlPr5}\\
u(x,t) &=& 0, \quad (x,t) \in \partial \Omega \times
  [0,T_e)\,,\label{CurlPr6}
\end{eqnarray}
where $n$ is a normal vector on the boundary $\partial\Omega$. For 
simplicity we consider only homogeneous boundary condition.
The elasticity tensor ${\mathbb C}[x]: {\cal S}^3 \rightarrow {\cal S}^3$ is a
linear, symmetric, uniformly positive definite mapping. The tensor
 ${\mathbb C}$ has measurable coefficients.
Classical linear isotropic hardening is included for $L\not=0$. We assume that the linear mapping
$L:{\mathbb R}^N\to{\mathbb R}^N$ is positive semi-definite and satisfies the inequality
\begin{eqnarray}
(Lz,z)\ge\alpha\|\gamma\|^2,\hspace{3ex}z=(p,\gamma)\in{\mathbb R}^N, \label{IsotropStuff}  
\end{eqnarray}
for some positive constant $\alpha\in{\mathbb R}$.
In the model equations, the nonlocal
 backstress contribution is given by the dislocation density motivated term 
$\Sigma^{\rm lin}_{\rm curl}=- C_1B^T \Curl\Curl p$ together with 
the corresponding micro-hard boundary condition (\ref{CurlPr5}).
For the model we require that the nonlinear constitutive mapping
$(v\mapsto g(\cdot, v)):{\mathbb R}^N \rightarrow 2^{{\mathbb R}^N}$ is monotone\footnote{Here 
$ 2^{{\mathbb R}^N}$ denotes the power set of ${\mathbb R}^N$.}, i.e. 
 it satisfies
\begin{eqnarray}
0 \leq (v_{1}-v_{2})\cdot (v^{*}_{1}-v^{*}_{2}), \label{monotype2}  
\end{eqnarray}
for all $v_i \in {\mathbb R}^N,\ v^{*}_i \in g(x,v_i),\ i=1,2$, and for a.e. $x\in\Omega$.  We also require
that
\begin{eqnarray}
0 \in g(x,0), \hspace{3ex}\text{a.e.}\ x\in{\Omega}.  \label{monotype1} 
\end{eqnarray}
The mapping $x\mapsto g(x, \cdot): \Omega \rightarrow 2^{{\mathbb R}^N}$ is measurable
(see \cite{Castaing77,Hu97,Pankov97} for the definition of 
the measurability of multi-valued maps).
Given are the volume force $b(x,t)\in \cR^3$. In this work we also assume that the function 
$g$ possesses
{\bf the self-controlling property}, i.e. there exists a continuous
function ${\cal F}: {\mathbb R}_+\times{\mathbb R}_+\to{\mathbb R}_+$ such that the
inequality 
\begin{eqnarray} 
 \|Bg(y)\|_2\le{\cal F}\left( \|Lg(y)\|_2,  \|y\|_2 \right)
 \label{SelfControlling}
\end{eqnarray}
holds for all $y\in L^2(\Omega, {\mathbb R}^N)$. The self-controlling property was first introduced
 by Chelminski in \cite{Chelminski98} for the study of inelastic models of monotone type and beyond that class.
 \begin{Rem} Visco-plasticity is typically included in the former conditions by choosing the function $g$ to be in Norton-Hoff form, i.e. 
\begin{align*}
  g(\Sigma)=[|{\Sigma}|-\yieldlimit]_+^r\,\frac{\Sigma}{|\Sigma|}\, , \quad\Sigma\in{\cal M}^3\, ,
\end{align*}  
where $\yieldlimit$ is the flow stress and $r$ is some parameter together with $[x]_+:=\max(x,0)$. If $g:{\cal M}^3\mapsto {\cal S}^3$ then the flow is called irrotational (no plastic spin).

In case of a non-associative flow rule, $g$ is not a subdifferential but may e.g. be written as
\begin{align*}
     g(\Sigma)= \mathcal{F}_1(\Sigma) \, \partial\mathcal{F}_2(\Sigma)\, ,
\end{align*}
where $\mathcal{F}_1$ describes the yield-function and $\mathcal{F}_2$ the flow direction.
 \end{Rem}

\begin{Rem}
{\rm It is well known that classical viscoplasticity (without gradient effects)
 gives rise to a well-posed problem. We extend this result to our formulation
  of rate-dependent gradient plasticity. The presence of the
  classical linear isotropic hardening in our model
  is related to $L\not=0$ whereas 
  the presence of the nonlocal gradient term is always related to $C_1>0$. }
\end{Rem}
\paragraph{Notation.} Throughout the whole work we choose the numbers 
$q, q^*$ satisfying the
following conditions 
\[1 < q, q^* < \infty \ \  {\rm and}\ \
1/q + 1/q^* = 1,\] 
and $|\cdot|$ denotes a norm in ${\mathbb R}^k$, $k\in{\mathbb N}$. We also assume for simplicity that $\Gamma_{\rm hard}=\partial\Omega$.
Moreover, the following notations are used in this work. 
The space $W^{m,q}(\Omega, \cR^k)$ with $q \in [1, \infty]$ consists
of all functions in $L^q(\Omega, \cR^k)$ with weak derivatives in
$L^q(\Omega, \cR^k)$ up to order $m$. If $m$ is not integer, then $W^{m,q}(\Omega, \cR^k)$ denotes 
the corresponding Sobolev-Slobodecki space. 
We set $H^m(\Omega, \cR^k)= W^{m,2}(\Omega, \cR^k)$. The norm in
$W^{m,q}(\Omega, {\mathbb R}^k)$ is denoted by $\| \cdot \|_{m,q,\Omega}$
($\| \cdot \|_{q}:=\| \cdot \|_{0,q,\Omega}$). 
The operator $\Gamma_0$ defined by
\[\Gamma_0: v\in W^{1,q}(\Omega, \cR^k)\mapsto W^{1-1/q,q}(\partial\Omega, {\mathbb R}^k)\]
  denotes the usual trace operator. The space $W^{m,q}_0(\Omega, \cR^k)$ with $q \in [1, \infty]$ consists
of all functions $v$ in $W^{m,q}(\Omega, \cR^k)$ with $\Gamma_0v=0$.
One can define the bilinear 
form on the product space
$L^{q}(\Omega, {\cal M}^3)$$\times$$L^{q^*}(\Omega, {\cal M}^3)$ by
\[
(\xi, \zeta )_{\Omega} = \int_{\Omega} \xi(x) \cdot \zeta(x) dx.
\]
 The space
\[L^q_{\Curl}(\Omega, {\cal M}^3)=\{v\in L^q(\Omega, {\cal M}^3)\mid
\Curl v\in L^q(\Omega, {\cal M}^3)\}\]
is a Banach space with respect to the norm
\[\|v\|_{q,\Curl}=\|v\|_{q}+\|\Curl v\|_{q}.\]
The well known result on the generalized trace operator can be easily
adopted to the functions with values in ${\cal M}^3$ 
(see \cite[Section II.1.2]{Sohr01}). Then, according
to this result, there is a bounded operator 
$\Gamma_n$ on $L^q_{\Curl}(\Omega, {\cal M}^3)$ 
\[\Gamma_n: v\in L^q_{\Curl}(\Omega, {\cal M}^3)
\mapsto\big(W^{1-1/{q^*},q^*}(\partial\Omega, {\cal M}^3)\big)^{*}\]
with \[\Gamma_n v=v\times n\big|_{\partial\Omega}\ {\rm if}\
v\in C^1(\bar{\Omega}, {\cal M}^3),\]
where $X^*$ denotes the dual of a Banach space $X$.
 Next,
\[L^{q}_{\Curl,0}(\Omega, {\cal M}^3)=
\{w\in L^{q}_{\Curl}(\Omega,{\cal M}^3)\mid \Gamma_n(w)=0\}.\]
We also define the space 
$Z^q_{\Curl}(\Omega, {\cal M}^3)$ by
\[Z^q_{\Curl}(\Omega, {\cal M}^3)=\{v\in L^q_{\Curl,0}(\Omega, {\cal M}^3)\mid
\Curl\Curl v\in L^q(\Omega, {\cal M}^3)\},\]
which is a Banach space with respect to the norm 
$$\|v\|_{Z^q_{\Curl}}=\|v\|_{q,\Curl}+\|\Curl\Curl v\|_{q}.$$ 
For
functions $v$ defined on $\Omega \times [0,\infty)$ we denote by
$v(t)$ the mapping $x \mapsto v(x,t)$, which is defined on $\Omega$.
 The space $L^q(0,T_e; X)$ denotes the Banach space of all Bochner-measurable 
functions $u:[0,T_e)\to X$ such that $t\mapsto\|u(t)\|^q_X$ is integrable
on $[0,T_e)$. Finally, we frequently use the spaces $W^{m,q}(0,T_e;X)$, 
which consist of Bochner measurable functions having $q$-integrable weak
derivatives up to order $m$.

\section{Preliminaries}\label{Operproperties}

\paragraph{Some properties of the $\Curl\Curl$-operator}
In this subsection we present some results concerning the $\Curl\Curl$-operator, which 
are relevant to the further investigations. For the $\Curl\Curl$-operator
with a slightly different domain of definition similar results are obtained in 
\cite[Section 4]{NesenenkoNeff2011}. Here we adopt the results in 
\cite{NesenenkoNeff2011} to our purposes. 
\begin{Lem}[Self-adjointness of $\Curl\Curl$]\label{SelfadjointCurlCurl} 
Let $\Omega\subset \cR^3$ be an open bounded set with a Lipschitz boundary and
$A:L^2(\Omega, {\cal M}^3)\to L^2(\Omega, {\cal M}^3)$ 
be the linear operator defined by
\[Av=\Curl\Curl v\]
with $\dom(A)=Z^2_{\Curl}(\Omega,{\cal M}^3)$. The operator $A$ is selfadjoint and
non-negative.
\end{Lem}
\begin{proof} Indeed, let us consider first the following linear closed operator
$S:L^2(\Omega, {\cal M}^3)\to L^2(\Omega, {\cal M}^3)$ defined by
\[Sv=\Curl v, \hspace{5ex} v\in\dom(S)=L^2_{\Curl,0}(\Omega,{\cal M}^3).\]
It is easily seen that its adjoint is given by
\[S^*v=\Curl v, \hspace{5ex} v\in\dom(S^*)=L^2_{\Curl}(\Omega,{\cal M}^3).\]
Then, by Theorem V.3.24 in \cite{Kato66}, the operator $S^*S$ with
\[\dom(S^*S)=\{v\in\dom(S)\mid Sv\in\dom(S^*)\},\] which is exactly the
operator $A$, is self-adjoint in $L^2(\Omega,{\cal M}^3)$. The non-negativity of
$A$ follows from its representation by the operator $S$, i.e. $A=S^*S$,
and the identity
\[(Av,u)_\Omega=(S^*Sv,u)_\Omega=(Sv,Su)_\Omega,\]
which holds for all $v\in\dom(A)$ and $u\in\dom(S)$.
\end{proof}
\begin{Corol}\label{MaxMonCurlCurl}
The operator $A:L^2(\Omega, {\cal M}^3)\to L^2(\Omega, {\cal M}^3)$ defined in 
Lemma~\ref{SelfadjointCurlCurl} is 
maximal monotone.
\end{Corol}
\begin{proof}
According to the result of Brezis (see \cite[Theorem 1]{Brez70}), 
a linear monotone operator  $A$ is
 maximal monotone, if it is a densely defined closed  operator
 such that its adjoint $A^*$ is monotone. The statement of the corollary follows then
directly from Lemma~\ref{SelfadjointCurlCurl} and the mentioned result of Brezis.
\end{proof}
\paragraph{Boundary value problems.}
Let $\Omega\subset \cR^3$ be an open bounded set with a Lipschitz boundary.
For every $v\in L^2(\Omega, {\cal M}^3)$ we define a functional $\Psi$ on
$L^2(\Omega, {\cal M}^3)$ by
\[\Psi(v)=\begin{cases}\frac12\int_\Omega|\Curl v(x)|^2dx, & v\in 
                                      L^2_{\Curl,0}(\Omega,{\cal M}^3)\\
                       +\infty, & {\rm otherwise}
  \end{cases}.\]
It is easy to check that $\Psi$ is proper, convex, lower semi-continuous.
The next lemma gives a precise description of the subdifferential 
$\partial\Psi$.
\begin{Lem} \label{CurlCurl}
We have that $\partial\Psi=\Curl\Curl$ with
\[\dom(\partial\Psi)=Z^2_{\Curl}(\Omega,{\cal M}^3).\]
\end{Lem}
\begin{proof} Let $A:L^2(\Omega, {\cal M}^3)\to L^2(\Omega, {\cal M}^3)$ 
be the linear
operator defined by
\[Av=\Curl\Curl v\]
and $\dom(A)=Z^2_{\Curl}(\Omega,{\cal M}^3)$. Due to Lemma~\ref{SelfadjointCurlCurl}, 
the following identity 
\begin{eqnarray}
\label{monotonyCurlCurl}\int_\Omega\Curl\Curl v(x)\cdot w(x)dx=
\int_\Omega\Curl v(x)\cdot\Curl w(x)dx
\end{eqnarray}
holds for any $v,w\in Z^2_{\Curl}(\Omega,{\cal M}^3)$. 
Therefore, using (\ref{monotonyCurlCurl}) we obtain
\[\int_\Omega\Curl\Curl v\cdot(w-v)dx=\int_\Omega\Curl v\cdot(\Curl w-\Curl v)dx\le\Psi(w)-\Psi(v)\]
for every $v,w\in\dom(A)$. This shows that
$A\subset\partial\Psi$. Since $A$ is maximal monotone 
(see Corollary~\ref{MaxMonCurlCurl}) we conclude that $A=\partial\Psi$.
\end{proof}
\paragraph{Helmholtz's projection.}
\label{HelmholtzProjectors}
In the linear elasticity theory it is well known 
(see \cite[Theorem 10.15]{Giusti2003}) 
that a Dirichlet boundary value problem formed by the equations
\begin{eqnarray}
\label{elast1} -\di_x \sigma(x) &=& \hat{b}(x), \hskip3.6cm \ x \in \Omega,\\
\label{elast2} \sigma(x) &=&  {\mathbb C}[x]( \sym
 {({ \nabla _x u(x)})}-{\hat\varepsilon}_p(x)),  \quad  \quad x \in \Omega,\\
 \label{elast3} u(x) &=&0, \hskip3.5cm  \ x \in {
\partial \Omega},
\end{eqnarray}
to given $\hat{b} \in W^{-1,q}( \Omega, {\mathbb R}^3)$ and
${\hat \varepsilon}_p \in L^q( \Omega, {\cal S}^3)$ has a unique
weak solution $( u, \sigma ) \in  W^{1,q}_0( \Omega, {\mathbb R}^3)
 \times L^q ( \Omega, {\cal S}^3)$ provided the open set $\Omega$
has a $C^1$-boundary and ${\mathbb C}$ is continuous on $\bar\Omega$. 
Here the number $q$ satisfies $1<q<\infty$. The
solution of (\ref{elast1}) - (\ref{elast3}) satisfies the inequality
\[ \|   \sym ( \nabla _x u ) \|_{q} 
\leq  C ( \|  \hat {\varepsilon }_p  \|_{q}+\|  \hat b  \|_{-1,q})\]
with some positive constant $C$. 
\begin{Def}\label{defProj}
For every $\hat{\varepsilon}_p \in L^q(\Omega , {\cal S}^3)$ we
define a linear operator $P_q : L^q(\Omega, {\cal S}^3) \to
L^q(\Omega , {\cal S}^3)$ by
\[P_q \hat{\varepsilon}_p := \sym ( \nabla _x u), \]
where $u \in W_0^{1,q}(\Omega , \cR^3)$ is the unique weak solution of
(\ref{elast1}) - (\ref{elast3}) to the given function
$\hat{\varepsilon}_p$ and $\hat{b}=0$.
\end{Def}
Next, a subset ${\cal G}^q$ of $L^q(\Omega, {\cal S}^3)$ is defined by
\[{\cal G}^q= \{\sym ( \nabla _x u)\ | \ 
u \in W_0^{1,q}(\Omega , \cR^3)\}. \]
The main properties of $P_q$ are stated in the following lemma.
\begin{Lem}\label{LemmaProj}
For every $1 < q < \infty$ the operator $P_q$ is a bounded
projector onto the subset ${\cal G}^q$ of $L^q(\Omega, {\cal
S}^3)$. The projector $(P_q)^*$ adjoint with respect to the
bilinear form $[ \xi , \zeta ]_{\Omega}:=( {\mathbb C}\xi , \zeta )_{\Omega}$ on $L^q(\Omega, {\cal
S}^3) \times L^{q^*}(\Omega, {\cal S}^3)$ satisfies
\[ (P_q)^* = P_{q^*}, \ \ {where} \ \ \frac{1}{q^*} + \frac{1}{q} = 1. \]
\end{Lem}
Due to Lemma~\ref{LemmaProj}  the following projection operator
\[ Q_q = (I - P_q): L^q(\Omega, {\cal S}^3) \to L^q(\Omega, {\cal S}^3) \]
 is well-defined and generalizes the classical Helmholtz projection.

Let $L:{\mathbb R}^N\to{\mathbb R}^N$ be a linear, positive semi-definite mapping.
The next result is needed for the subsequent analysis.
\begin{Corol}\label{nonnegative}
The operator $B^T\sym{\mathbb C}Q_2\sym B+L:  L^2(\Omega, {\mathbb R}^N) \to {L^2}(\Omega, {\mathbb R}^N)$
is non-negative and self-adjoint.
\end{Corol}
For the proof of this result the reader is referred to \cite{AlCh02}.



\section{Existence of strong solutions}\label{Existence}

In this section we prove the main existence result for (\ref{CurlPr1}) - (\ref{CurlPr6}).
To show the existence of weak solutions a time-discretization
method is used in this work. In the first step, we prove the existence of the solutions of
the time-discretized problem in an appropriate Hilbert spaces based on the
Helmholtz projection in $L^2(\Omega, {\cal S}^3)$ (Section~\ref{HelmholtzProjectors}) 
and the monotone operator methods. In order to be able to apply 
the monotone operator method to the time-discretized problem we regularize it
by a linear positive definite term. In the second step,
we derive the uniform a priori estimates for the solutions of the time-discretized problem
using the polynomial growth of
 the function $g$ (see Definition~\ref{CoercClass} below) and then 
 we pass to the weak limit in the
equivalent formulation of the time-discretized problem employing the weak lower
semi-continuity of lower semi-continuous convex functions and the maximal monotonicity
of $g$.
\paragraph{Main result.} First, we define a class of maximal monotone functions we deal with
in this work.
\begin{Def}\label{CoercClass}
For $m\in L^1(\Omega, \cR)$, $\alpha\in{\mathbb R}_{+}$ and $q>1$, 
${\mathbb M}(\Omega,{\mathbb R}^k,q,\alpha,m)$ is the set of 
multi-valued functions $h:\Omega \times {\mathbb R}^k\rightarrow 2^{\cR^k}$
with the following properties
\begin{itemize}
    \item $v \mapsto h(x, v)$ is maximal monotone for almost all $x \in \Omega$,
    \item the mapping $x \mapsto j_{\lambda}(x, v) : \Omega \rightarrow \cR^k$ is
measurable for all $\lambda > 0$, where $j_{\lambda}(x, v)$ is
     the inverse of $v \mapsto v + \lambda h(x, v)$,
    \item for a.e. $x\in \Omega$ and every $v^*\in h(x,v)$
     \begin{eqnarray}
\label{inequMain} \alpha\left(\frac{|v|^q}{q}+\frac{|v^*|^{q^*}}{q^*}\right)
\le (v,v^*)+m(x),
\end{eqnarray}
where $1/q+1/{q^*}=1$.
\end{itemize}
\end{Def}
\begin{Rem}{\rm
We note that the condition \eq{inequMain} is equivalent to the following
two inequalities
\begin{eqnarray}
\label{inequMain1} &&|v^*|^{q^*}\le m_1(x)+\alpha_1|v|^q,\\
&& (v,v^*)\ge m_2(x)+\alpha_2|v|^q,\label{inequMain2}
\end{eqnarray}
for a.e. $x\in \Omega$, every $v^*\in h(x,v)$, with suitable functions
$m_1,m_2\in L^1(\Omega, \cR)$ and numbers 
$\alpha_1,\alpha_2\in{\mathbb R}_{+}$.} 
\end{Rem}
The main properties of the class ${\mathbb M}(\Omega,{\mathbb R}^k,q,\alpha,m)$ are
collected in the following proposition.
\begin{Prop}\label{MainClassMaxMonoProp}
Let ${\cal H}$ be a canonical extension\footnote{
 The canonical extension ${\cal H}$ of a function
 $h:\Omega\to{\mathbb M}(\Omega,{\mathbb R}^k,q,\alpha,m)$ is a monotone graph
 from $L^q(\Omega,{\mathbb R}^k)$ to $L^{q^*}(\Omega,{\mathbb R}^k)$, with
$1/q + 1/{q^*} = 1$, defined by:
\[Gr {\cal H} = \{[v, v^*] \in L^q(\Omega,{\mathbb R}^k)\times L^{q^*}(\Omega,{\mathbb R}^k)\mid
 [v(x), v^*(x)] \in Gr\ h(x)\ for\ a.e.\ x\in \Omega\}.\]
}
 of a function $h:\Omega\times{\mathbb R}^k\to 2^{{\mathbb R}^k}$,
which belongs to
    ${\mathbb M}(\Omega,{\mathbb R}^k,q,\alpha,m)$. Then ${\cal H}$
  is maximal monotone, surjective and $D({\cal H})=L^q(\Omega,{\mathbb R}^k)$.
     \end{Prop}
\begin{proof}
 See Corollary 2.15  in \cite{Damlamian07}. 
\end{proof}
Next, we define
the following notion of solutions for 
the initial boundary value problem (\ref{CurlPr1}) - (\ref{CurlPr6}). Both notions of the solutions
for (\ref{CurlPr1}) - (\ref{CurlPr6}) are introduced without assuming the homogeneity
of the the initial condition (\ref {CurlPr4}).
\begin{Def}\label{StrongSol}{\bf (Strong solutions)}
 A function $(u,\sigma,z)$ with $z=(p, \gamma)$ such that
\[(u,\sigma)\in H^{1}(0,T_e; H^{1}_0(\Omega, {\mathbb R}^3) \times
 L^{2} ({\Omega}, {\cal S}^3)),\] \[ p\in   
H^{1}(0,T_e;L^{2}_{\Curl}(\Omega,{\cal M}^3))\cap 
L^{2}(0,{T_e};Z^{2}_{\Curl}(\Omega,{\cal M}^3))\]
and \[ \gamma\in   
H^{1}(0,T_e;L^{2}(\Omega,{\mathbb R})),\ \ 
\Sigma^{\rm lin}\in L^{q}(\Omega_{T_e}, {\cal M}^3)\]
is called a {\rm strong solution} of the initial boundary value
problem (\ref{CurlPr1}) - (\ref{CurlPr6}), if for every $t\in [0,T_e]$
the function $(u(t),\sigma(t))$ is a weak solution of the boundary value problem
 (\ref{elast1}) - (\ref{elast3}) with ${\hat\varepsilon}_p=\sym p(t)$ and
$\hat{b}=b(t)$, the evolution inclusion (\ref{CurlPr3}), the initial condition (\ref{CurlPr4}) and the boundary condition (\ref{CurlPr5}) are satisfied pointwise.
\end{Def}
 Next, we state the main result of this work.
\begin{Theor}\label{existMain} 
Suppose that $1<q^*\le 2\le q<\infty$. Assume that  
$\Omega \subset \cR^3$ is a sliceable domain with a $C^1$-boundary and the tensor
 ${\mathbb C}$ has $L^\infty$-coefficients.
Let the function
$b \in W^{1,q}(0,T_e; L^{q}(\Omega, {\mathbb R}^3))$ be given.
 Assume that the function $g\in{\mathbb M}(\Omega, {\mathbb R}^N, q, \alpha, m)$ is of a subdifferential type,
  enjoys the self-controlling property (\ref{SelfControlling})
  and  that for a.e. $x\in\Omega$
 the relation
 \begin{eqnarray}
\label{Assumption1} 0\in g(x, B^T\sigma^{(0)}(x))
\end{eqnarray}
holds, where the function 
$\sigma^{(0)}\in L^2({\Omega}, {\cal S}^3)$ is  determined by equations
(\ref{elast1}) - (\ref{elast3}) for ${\hat\varepsilon}_p=0$ and $\hat{b}=b(0)$.
Then there exists a strong solution $(u,\sigma,z)$
of the initial boundary value problem (\ref{CurlPr1}) - (\ref{CurlPr6}).
\end{Theor} 
In order to deal with the measurable elasticity tensor ${\mathbb C}$, we reformulate
the problem (\ref{CurlPr1}) - (\ref{CurlPr6}) as follows:\\
Let the function
 $(\hat v,\hat\sigma)\in W^{1,q}(0, T_e, W^{1,q}_0( \Omega, {\mathbb R}^3)
 \times L^q ( \Omega, {\cal S}^3))$ be a solution of the linear elasticity
 problem
  formed by the equations
\begin{eqnarray}
\label{elast1A} -\di_x \hat\sigma(x, t) &=& {b}(x, t), \hskip2.8cm \ x \in \Omega,\\
\label{elast2A} \hat\sigma(x, t) &=&  \hat{\mathbb C}( \sym
 ( \nabla _x \hat v(x, t)),  \quad \quad \quad x \in \Omega,\\
 \label{elast3A} \hat v(x, t) &=&0, \hskip3.5cm  \ x \in {
\partial \Omega},
\end{eqnarray}
where $\hat{\mathbb C}:{\cal S}^3\to{\cal S}^3$ is any positive definite symmetric
 linear mapping independent of $(x, t)$. Such a function  $(\hat v,\hat\sigma)$ exists
  (see \cite[Theorem 10.15]{Giusti2003}). Then the solution $(u, \sigma, z)$ of
  the initial boundary value problem (\ref{CurlPr1}) - (\ref{CurlPr6}) has the following
  form
  \[(u, \sigma, z)=(\tilde v+\hat v, \tilde \sigma+\hat\sigma, z),\]
  where the function $(\tilde v, \tilde\sigma, z)$ solves the problem
  \begin{eqnarray}
- \di_x \tilde\sigma(x,t) &=&  0, \label{CurlPr1A}
\\[1ex]
\tilde\sigma(x,t) &=& {\mathbb C}[x] ( \sym (\na_x \tilde v(x,t) - Bz(x,t) ) )  \label{CurlPr2A}\\[1ex]
&&\hspace{16ex}+({\mathbb C}[x] -\hat{\mathbb C})( \sym( \nabla _x \hat v(x, t)), \non\\[1ex] 
\label{CurlPr3A} \partial_t z(x,t) & \in & g \big(x,\Sigma^{\rm lin}(x,t)\big),  
\hspace{3ex} \Sigma^{\rm lin}=\Sigma^{\rm lin}_e+\Sigma^{\rm lin}_{\rm sh}
+\Sigma^{\rm lin}_{\rm curl},\label{microPr3A} \\[1ex] 
\Sigma^{\rm lin}_{\rm e}&=&B^T(\tilde\sigma+\hat\sigma),
\hspace{1ex} \Sigma^{\rm lin}_{\rm sh}=-Lz,
\hspace{1ex} \quad \Sigma^{\rm lin}_{\rm curl}=- C_1B^T \Curl\Curl Bz\, ,\non\\[1ex]
z(x,0) &=& 0, \quad\quad\quad x \in \Omega,  \label{CurlPr4A}   \\
Bz(x,t)\times n (x)&=&0,\hspace{7ex} (x,t) \in \partial\Omega \times
  [0,T_e), \label{CurlPr5A}\\
\tilde v(x,t) &=& 0, \quad\quad (x,t) \in \partial \Omega \times
  [0,T_e)\,.\label{CurlPr6A}
\end{eqnarray}
Here, the function $(\hat v,\hat\sigma)$ given as the solution of (\ref{elast1A}) - (\ref{elast3A})
is considered as known. Next, we show that the problem (\ref{CurlPr1A}) - (\ref{CurlPr6A}) 
has a solution. This will prove the existence of solutions for (\ref{CurlPr1}) - (\ref{CurlPr6}).
\begin{proof} We show the existence
of  solutions using the Rothe method (a time-discretization
method, see \cite{Roubi05} for details).  
In order to introduce a time-discretized problem, let us fix any
$m\in{\mathbb N}$ and set
\[h:=\frac{T_e}{2^m}, \ z^0_m:=0,\ \hat\sigma^n_m:=\frac{1}{h}\int^{nh}_{(n-1)h}\hat\sigma(s)ds\in 
L^{q}( \Omega, {\cal S}^3),\ \ n=1,...,2^m. \]
Then we are looking for functions $u^n_m\in H^1(\Omega,{\mathbb R}^3)$,
$\sigma^n_m\in L^2(\Omega,{\cal S}^3)$ and $z^n_m\in
 Z^2_{\Curl}(\Omega,{\mathbb R}^N)$ with $Bz^n_m(x)\in\mathfrak{sl}(3)$ for a.e. $x\in\Omega$ and
\[\Sigma^{\rm lin}_{n,m}:=B^T(\sigma^n_m+\hat\sigma^n_m)-L z^n_m -
\frac{1}{m}z^n_m-C_1B^T \Curl\Curl Bz^n_m\in L^q(\Omega,{\mathbb R}^N)\]
solving the following problem
\begin{eqnarray}
- \di_x \sigma^n_m(x) &=&  0, \label{CurlPr1Dis}
\\[1ex]
\sigma^n_m(x) &=& {\mathbb C} [x]( \sym (\na_x u^n_m(x) - Bz^n_m(x) ) )   \label{CurlPr2Dis}\\[1ex]
&&+({\mathbb C}[x] -\hat{\mathbb C})(\hat{\mathbb C})^{-1}\hat\sigma^n_m(x), \non
\\[1ex] 
\label{CurlPr3Dis} \frac{z^n_m(x)-z^{n-1}_m(x)}{h} & \in & 
g \big(x, \Sigma^{\rm lin}_{n,m}(x)\big),\label{microPr3Dis} 
\end{eqnarray}
together with the boundary conditions 
\begin{eqnarray} 
Bz^n_m(x)\times n (x)&=&0,\hspace{7ex} x \in \partial\Omega, 
\label{CurlPr5Dis}\\
u^n_m(x) &=& 0, \quad x \in \partial \Omega\,.\label{CurlPr6Dis}
\end{eqnarray}
Next, we adopt the reduction technique proposed in \cite{AlCh02} to the above equations.
Suppose that $(u^n_m, \sigma^n_m, z^n_m)$ is a solution of the boundary value
problem (\ref{CurlPr1Dis}) - (\ref{CurlPr6Dis}). 
The equations 
(\ref{CurlPr1Dis}) - (\ref{CurlPr2Dis}),
(\ref{CurlPr6Dis})  form a boundary value problem for the solution $(u^n_m, \sigma^n_m)$ of the 
problem of linear elasticity. Due to linearity of this problem we can write these components
of the solution in the form
\[(u^n_m, \sigma^n_m) = ({U}^n_m, {\Sigma}^n_m) + 
(w^n_m,\tau^n_m),\]
with the solution $(w^n_m,\tau^n_m)$ of the Dirichlet boundary
value problem (\ref{elast1}) - (\ref{elast3}) to the data $\hat{b} =0$, 
$\hat{\varepsilon}_p = {\mathbb C}^{-1}({\mathbb C} -\hat{\mathbb C})(\hat{\mathbb C})^{-1}\hat\sigma^n_m$, and with the solution $({U}^n_m, {\Sigma}^n_m)$ of the
problem (\ref{elast1}) - (\ref{elast3}) to the data $\hat{b} =0$, 
$\hat{\varepsilon}_p = \sym(Bz^n_m)$. 
We thus obtain
\[  \sym ( \nabla _x u^n_m) - {\sym}(Bz^n_m) = (P_2
  - I) {\sym}(Bz^n_m) + \sym ( \nabla _x w^n_m). \]
where the operator $P_2$ is defined in Definition~\ref{defProj}.
We insert this equation into (\ref{CurlPr2Dis}) and get that
(\ref{CurlPr3Dis}) can be rewritten in the following form
\begin{eqnarray}
\label{RedCurl1}&& \frac{z^n_m-z^{n-1}_m}{h}  \in  
 {\cal G} \big(-M_mz^n_m- C_1B^T \Curl\Curl Bz^n_m
+B^T(\hat\sigma^n_m+\tau^n_m)\big),\\
\label{RedPr2}&& Bz^n_m(x)\times n (x)=0,
 \hspace{7ex} x \in \partial\Omega,
\end{eqnarray}
where 
$$M_m:=(B^T\sym{\mathbb C}Q_2\sym B+L) +\frac{1}{m}I:
L^2(\Omega,{\cal M}^3)\to L^2(\Omega,{\cal M}^3)$$ 
with  the Helmholtz projection $Q_2$.  Here ${\cal G}$ denotes the canonical
extension of $g$. Next,
 the problem (\ref{RedCurl1}) - (\ref{RedPr2}) reads 
\begin{eqnarray}\label{ImportantEquiv}
\Psi(z^n_m)\ni B^T(\hat\sigma^n_m+\tau^n_m),
\end{eqnarray}
where
\[\Psi(v)={\cal G}^{-1}\Big(\frac{v-z^{n-1}_m}{h}\Big)+M_m(v)+\partial\Phi(v).\]
Here, the functional $\Phi:L^2(\Omega,{\cal M}^3)\to\bar\cR$ is given by
\[\Phi(v):=
\begin{cases}\frac{1}{2}\int_\Omega|\Curl v(x)|^2dx, & v\in 
               L^2_{\Curl,0}(\Omega,{\cal M}^3)\\
                       +\infty, & {\rm otherwise}
  \end{cases},\]
respectively. The facts that $\Phi$ is a proper convex lower semi-continuous
functional and that $\Curl\Curl=\partial\Phi$ are proved 
in Section~\ref{Operproperties}. Since $M_m$ is bounded, self-adjoint and positive
definite  (see Corollary~\ref{nonnegative} and the definition of 
$M_m$), it is maximal monotone by 
Theorem II.1.3 in \cite{Barb76}. The last thing
 which we
have to verify is whether the following operator
\[\Psi={\cal G}^{-1}+M_m+\partial\Phi\]
is maximal monotone. Since $g\in{\mathbb M}(\Omega, {\cal M}^3, q, \alpha, m)$, using the boundness of $M_m$
we conclude that the domains of ${\cal G}^{-1}$ and $M_m$
are equal to the whole space $L^2(\Omega,{\cal M}^3)$.
Therefore, Theorem III.3.6 in \cite{Pas78} (or \cite[Theorem II.1.7]{Barb76})
 guarantees that the sum ${\cal G}^{-1}+M_m+\partial\Phi$
is maximal monotone with 
\[\dom(\Psi)=\dom(\partial\Phi)
:=Z^2_{\Curl}(\Omega,{\cal M}^3).\]
Since $M_m$ is coercive in $L^2(\Omega,{\cal M}^3)$, 
which obviously yields the coercivity of $\Psi$, 
 the operator
$\Psi$ is surjective by Theorem III.2.10 in \cite{Pas78}.
Thus, we conclude that equation (\ref{ImportantEquiv})
as well as
the problem (\ref{RedCurl1}) - (\ref{RedPr2}) has a solution with the 
required regularity, i.e. $Bz^n_m\in
 Z^2_{\Curl}(\Omega,{\cal M}^3)$. The solution is unique due to the following estimate
\begin{eqnarray}\label{est_n}
0=(\Psi (z_1)-\Psi (z_2), z_1-z_2)_\Om\geq (M_m( z_1-z_2),z_1-z_2)_\Om\geq \frac 1m\|z_1-z_2\|^2_\Om,
\end{eqnarray} 
which holds since the operators ${\cal G}^{-1}$ and $\partial\Phi$ are monotone.
By the constructions this implies that the boundary value problem
(\ref{CurlPr1Dis}) - (\ref{CurlPr6Dis}) is solvable as well (details
can be found in \cite{AlCh02}). Moreover, 
$Bz^n_m(x)\in\mathfrak{sl}(3)$ for a.e. $x\in\Omega$.
\vspace{1ex}\\
{\bf Rothe approximation functions:} 
For any family $\{\xi^{n}_m\}_{n=0,...,2m}$ of functions in a reflexive Banach
space $X$, we define
{\it the piecewise affine interpolant} $\xi_m\in C([0,T_e],X)$ by
\begin{eqnarray}\label{RotheAffineinterpolant}
\xi_m(t):= \left(\frac{t}h-(n-1)\right)\xi^{n}_m+
\left(n-\frac{t}h\right)\xi^{n-1}_m \ \ {\rm for} \ (n-1)h\le t\le nh
\end{eqnarray}
and {\it the piecewise constant interpolant} $\bar\xi_m\in L^\infty(0,T_e;X)$ by
\begin{eqnarray}\label{RotheConstantinterpolant}
\bar\xi_m(t):=\xi^{n}_m\  {\rm for} \ (n-1)h< t\le nh, \ 
 n=1,...,2^m, \ {\rm and} \ \bar\xi_m(0):=\xi^{0}_m.
\end{eqnarray}
For the further analysis we recall the following property of $\bar\xi_m$
and $\xi_m$: 
\begin{eqnarray}\label{RotheEstim}
\|\xi_m\|_{L^q(0,T_e;X)}\le\|\bar\xi_m\|_{L^q(-h,T_e;X)}\le
\left(h \|\xi^{0}_m\|^q_X+\|\bar\xi_m\|^q_{L^q(0,T_e;X)}\right)^{1/q},
\end{eqnarray}
where $\bar\xi_m$ is formally extended to $t\le0$ by $\xi^{0}_m$ and
$1\le q\le\infty$ (see \cite{Roubi05}).\vspace{1ex}\\
{\bf A-priori estimates.} 
Multiplying (\ref{CurlPr1Dis}) by $(u^n_m-u^{n-1}_m)/h$ and then integrating
over $\Omega$ we get
\[\big(\sigma^n_m, \sym (\na_x (u^n_m-u^{n-1}_m)/h)\big)_\Omega=
0.\]
Equations (\ref{CurlPr2Dis}), (\ref{CurlPr3Dis}) imply that for a.e. $x\in\Omega$
\[\sigma^n_m\cdot\Big(\sym (\na_x(u^n_m-u^{n-1}_m)/h)-{\mathbb C}^{-1}[x]
(\sigma^n_m-\sigma^{n-1}_m)/h\Big)\]
\[+\sigma^n_m\cdot\Big({\mathbb C}^{-1}[x]({\mathbb C}[x] -\hat{\mathbb C})(\hat{\mathbb C})^{-1}
(\hat\sigma^n_m-\hat\sigma^{n-1}_m)/h\Big)\]
\[-\frac{z^n_m-z^{n-1}_m}h\cdot\Big(Lz^n_m
+\frac{1}{m} z^n_m+ C_2B^T \Curl\Curl Bz^n_m\Big)\]
\[+\frac{z^n_m-z^{n-1}_m}h\cdot B^T\hat\sigma^n_m=g^{-1}\Big(\frac{z^n_m-z^{n-1}_m}h\Big)
\cdot\Big(\frac{z^n_m-z^{n-1}_m}h\Big).\]
After integrating the last identity
over $\Omega$,  the above computations imply 
\[\frac{1}{h}\Big({\mathbb C}^{-1}\sigma^n_m,
\sigma^n_m-\sigma^{n-1}_m\Big)_\Omega+ \frac{1}{h}
\Big(L^{1/2} (z^n_m-z^{n-1}_m),L^{1/2} z^n_m\Big)_\Omega\]
\[
+\frac{1}{m} \frac{1}{h}\Big(z^n_m-z^{n-1}_m, z^n_m\Big)_\Omega
+ C_1 \frac{1}{h} \Big(\Curl B(z^n_m-z^{n-1}_m),\Curl B z^n_m\Big)_\Omega
\]\[
+\frac{\alpha}{q}\left\| \Sigma^{\rm lin}_{n,m}\right\|^{q}_{q}
+\frac{\alpha}{q^*}\left\|\frac{z^n_m-z^{n-1}_m}h\right\|^{q^*}_{q^*}\le\int_\Omega m(x)dx\]
\[
+\frac{1}{h}\left(\sigma^n_m, \bar{\mathbb C}(\hat\sigma^n_m-\hat\sigma^{n-1}_m)\right)_\Omega+\frac{1}{h}\left(B^T\hat\sigma^n_m, z^n_m-z^{n-1}_m\right)_\Omega,\]
where 
$\bar{\mathbb C}:={\mathbb C}^{-1}({\mathbb C} -\hat{\mathbb C})(\hat{\mathbb C})^{-1}$.
Multiplying by $h$ and summing the obtained relation for $n=1,...,l$ 
for any fixed $l\in[1,2^m]$ we derive the following inequality 
(here ${\mathbb B}:={\mathbb C}^{-1}$)
\begin{eqnarray}
&& \frac{1}2\Big(
\|{\mathbb B}^{1/2}\sigma^l_m\|^2_2+\|L^{1/2} z^l_m\|^2_2
+\frac{1}{m}\|z^l_m\|^2_2+ C_1\|\Curl Bz^l_m\|^2_2\Big)
\non\\
&&\label{AprioriEstimHelp1}+\frac{h\alpha}{q}\sum^l_{n=1}\left\| \Sigma^{\rm lin}_{n,m}\right\|^{q}_{q}
+\frac{h\alpha}{q^*}\sum^l_{n=1}\left\|\frac{z^n_m-z^{n-1}_m}h\right\|^{q^*}_{q^*}\le
C^{(0)}+lh\int_\Omega m(x)dx\\&&+
h\sum^l_{n=1}\left(\sigma^n_m,\bar{\mathbb C}
\frac{\hat\sigma^n_m-\hat\sigma^{n-1}_m}h\right)_\Omega
+
h\sum^l_{n=1}\left(B^T\hat\sigma^n_m,\frac{z^n_m-z^{n-1}_m}h\right)_\Omega,\non
\end{eqnarray}
where\footnote{Here we use the following inequality
\[\sum^l_{n=1}(\phi^n_m-\phi^{n-1}_m,\phi^n_m)_\Omega=
\frac{1}2\sum^l_{n=1}\Big(\|\phi^n_m\|^2_2
-\|\phi^{n-1}_m\|^2_2\Big)\]
\[+\frac{1}2\sum^l_{n=1}\|\phi^n_m-\phi^{n-1}_m\|^2_2
\ge\frac{1}2\|\phi^l_m\|^2_2
-\frac{1}2\|\phi^{0}_m\|^2_2\]
for any family of functions $\phi^{0}_m,\phi^{1}_m,...,\phi^{m}_m$.}
\[C^{(0)}:=
\|{\mathbb B}^{1/2}\sigma^{(0)}\|^2_2.\]
Since $\hat\sigma^n_m\in L^q ( \Omega, {\cal S}^3)$,
using Young's
inequality with $\epsilon>0$ we get that
\begin{eqnarray}\label{AprioriEstimHelp3} 
&&\left(B^T\hat\sigma^n_m,\frac{z^n_m-z^{n-1}_m}h\right)_\Omega\le \|B^T\hat\sigma^n_m\|_{q}
\|({z^n_m-z^{n-1}_m})/h\|_{q^*}\non\\
&&\hspace{14ex}\le C_\epsilon
 \|B^T\| \|\hat\sigma^n_m\|^q_{q}+\epsilon\|(z^n_m-z^{n-1}_m)/h\|^{q^*}_{q^*},
\end{eqnarray}
where $C_\epsilon$ is a positive constant appearing in the Young's
inequality. Analogically, we obtain
\begin{eqnarray}\label{AprioriEstimHelp3A} 
&&\left(\sigma^n_m, \bar{\mathbb C}\frac{\hat\sigma^n_m-\hat\sigma^{n-1}_m}h\right)_\Omega
\le
\epsilon\|\sigma^n_m\|^2_{2}+C_\epsilon\|(\hat\sigma^n_m-\hat\sigma^{n-1}_m)/h\|^{2}_{2}
\end{eqnarray}
with some other constant $C_\epsilon$.
Combining the inequalities (\ref{AprioriEstimHelp1}), 
(\ref{AprioriEstimHelp3}) and (\ref{AprioriEstimHelp3A}),
and choosing an appropriate value for $\epsilon>0$ we obtain 
the following estimate
\begin{eqnarray}\label{aprioriEstim1}
&&\frac{1}2\Big(
\|{\mathbb B}^{1/2}\sigma^l_m\|^2_2+\|L^{1/2} z^l_m\|^2_2
+\frac{1}{m}\|z^l_m\|^2_2+ C_1\|\Curl Bz^l_m\|^2_2\Big)\non\\
&&+h\hat{C}_\epsilon\sum^l_{n=1}\left(\left\| \Sigma^{\rm lin}_{n,m}\right\|^{q}_{q}+
\Big\|\frac{z^n_m-z^{n-1}_m}h\Big\|^{q^*}_{q^*}\right)
\le
C^{(0)}+lh\int_\Omega m(x)dx
\\&&+h\epsilon\sum^l_{n=1}\|\sigma^n_m\|^2_{2}
+h\tilde{C}_\epsilon\sum^l_{n=1}\Big(\|\hat\sigma^n_m\|^q_{q}+
\|(\hat\sigma^n_m-\hat\sigma^{n-1}_m)/h\|^{2}_{2}\Big),\non
\end{eqnarray}
where $\tilde{C}_\epsilon$ and $\hat{C}_\epsilon$ are some positive
constants. Now, taking Remark 8.15 in \cite{Roubi05} and the definition
of Rothe's approximation functions into account we 
rewrite (\ref{aprioriEstim1}) as follows
\begin{eqnarray}
&&
\|{\mathbb B}^{1/2}\bar\sigma_m(t)\|^2_2+\|L^{1/2}\bar z_m(t)\|^2_2
+\frac{1}{m}\|\bar z_m(t)\|^2_2+ C_1\|\Curl B\bar z_m(t)\|^2_2\non\\
&&\hspace{3ex}+
\hat{C}_\epsilon\int_0^{T_e}\int_\Omega\left( \big|\partial_t z_m(x,t)\big|^{q^*} 
+|\bar\Sigma^{\rm lin}_{m}(x,t)|^q\right) dxdt\label{aprioriEstim2} \\
&&\hspace{3ex}\le
2C^{(0)}+2T_e\|m\|_{1,\Omega}+\epsilon\|\sigma_m\|^2_{2,\Omega\times(0, T_e)}
+2\tilde{C}_\epsilon\|\hat\sigma\|^q_{W^{1,q}(0,{T_e};L^q(\Omega,{\cal S}^3))}.\non
\end{eqnarray}
From (\ref{aprioriEstim2}) we get immediately that
\begin{eqnarray}
&&
\bar{C}_\epsilon\|\sigma_m\|^2_{2,\Omega\times(0, t)}+\|L^{1/2}\bar z_m(t)\|^2_2
+\frac{1}{m}\|\bar z_m(t)\|^2_2+ C_1\|\Curl B\bar z_m(t)\|^2_2\non\\
&&\hspace{3ex}+
\hat{C}_\epsilon\left(\|\partial_t z_m\|^{q^*}_{q^*,\Omega\times(0, T_e)}+
\|\bar\Sigma^{\rm lin}_{m}\|^{q}_{q,\Omega\times(0, T_e)}\right)\label{aprioriEstim2A}
\\
&&\hspace{3ex}\le 2C^{(0)}+2T_e\|m\|_{1,\Omega}
+2\tilde{C}_\epsilon\|\hat\sigma\|^q_{W^{1,q}(0,{T_e};L^q(\Omega,{\cal S}^3))},\non
\end{eqnarray}
where $\bar{C}_\epsilon$ is some other constant depending on $\epsilon$.
Altogether, from estimate (\ref{aprioriEstim2A}) we get that
\begin{eqnarray}
&&\{z_m\}_m \ {\rm is}\  {\rm uniformly}\  {\rm bounded}\  {\rm in}
 \ W^{1,{q^*}}(0,{T_e};L^{q^*}(\Omega, {\mathbb R}^N)),
\label{aprioriEstim3}\\[1ex]
&&\{L^{1/2}\bar z_m\}_m \ {\rm is}\  {\rm uniformly}\  {\rm bounded}\  {\rm in}
 \ L^\infty(0,{T_e};L^{2}(\Omega,{\mathbb R}^N)),\label{aprioriEstim3new1}\\[1ex]
&&\{\sigma_m\}_m,\ {\rm is}\  {\rm uniformly}\  {\rm bounded}\  {\rm in}
 \ L^2(0,{T_e};L^{2}(\Omega,{\cal S}^3)),\label{aprioriEstim3a}\\[1ex]
&&\{\Curl B\bar z_m\}_m\ {\rm is}\  {\rm uniformly}\  {\rm bounded}\  {\rm in}
 \ L^\infty(0,{T_e};L^{2}(\Omega,{\cal M}^3)),\label{aprioriEstim3new2}\\[1ex]
&&\left\{\bar\Sigma^{\rm lin}_{m}\right\}_m\ {\rm is}\  {\rm uniformly}\  {\rm bounded}\  {\rm in}
 \ L^q(0,{T_e};L^{q}(\Omega,{\mathbb R}^N)),\label{aprioriEstim4cc}\\[1ex]
&&\left\{\frac1{\sqrt{m}} \bar z_m\right\}_m\ {\rm is}\  {\rm uniformly}\  {\rm bounded}\  {\rm in}
 \ L^\infty(0,{T_e};L^{2}(\Omega, {\mathbb R}^N)).\label{aprioriEstim3b} \end{eqnarray}
In particular, the uniform boundedness of the sequences in (\ref{aprioriEstim3}) - (\ref{aprioriEstim3b}) yields 
\begin{eqnarray}
&&\{u_m\}_m \ {\rm is}\  {\rm uniformly}\  {\rm bounded}\  {\rm in}
 \  W^{1,{q^*}}(0,{T_e};W^{1,{q^*}}_0(\Omega,{\mathbb R}^3)),\\[1ex]
&&\{\Curl\Curl B\bar z_m\}_m\ {\rm is}\  {\rm uniformly}\  {\rm bounded}\  
{\rm in} \ L^2(0,{T_e};L^{2}(\Omega,{\cal M}^3)).\label{aprioriEstim4}
\end{eqnarray}
Employing (\ref{RotheEstim}) the estimates (\ref{aprioriEstim3}) - 
(\ref{aprioriEstim4}) further imply that
 the sequences  $\{\sigma_m\}_m$, $\{L^{1/2} z_m\}_m$,
$\{\Curl Bz_m\}_m$, $\left\{ z_m/\sqrt{m}\right\}_m$, 
$\left\{\Sigma^{\rm lin}_{m}\right\}_m$ and $\{\Curl\Curl Bz_m\}_m$ are also
uniformly bounded in the corresponding spaces. As a result, we have
\begin{eqnarray}
\{Bz_m\}_m \ {\rm is}\  {\rm uniformly}\  {\rm bounded}\  {\rm in}
 \ L^{q^*}(0,{T_e};Z^{q^*}_{\Curl}(\Omega,{\cal M}^3)).
\label{aprioriEstim6}
\end{eqnarray}

Using the assumption that $g$ is of subdifferential type, i.e. there exists a function $\tilde f$ such that
$$(g(x,v),v-w)\geq \tilde f(x,v)-\tilde f(x,w)$$ is satisfied for a.e. $x\in\Om$ and all $w\in \R^N$, we can improve the estimates (\ref{aprioriEstim3new1}), (\ref{aprioriEstim3new2}) and  (\ref{aprioriEstim3b}). By assumption (\ref{Assumption1}) we obtain for a.e. $x\in\Om$ and all $w\in L^2(\Om,\R^N)$
$$\tilde f(x,w(x))\geq \tilde f(x,B^T\sigma^{(0)}(x)).$$
This inequality implies that $\tilde f$ attains its minimum at the point $B^T\sigma^{(0)}$. We introduce the function $f$ defined by
\begin{align}
f(x,v)=\tilde f(x,v)-\tilde f(x,B^T\sigma^{(0)}(x)).
\end{align}
Then the functional $f$ satisfies $f\geq 0$ and $g=\partial\tilde f=\partial f$ and we can suppose that $g(x,v)=\partial f(x,v)$ with $f\geq 0$.

Next, we show that the sequence $\left\{f(\Sigbar_{m})\right\}_m$ is uniformly bounded 
in the space $L^\infty(0,{T};L^{1}(\Om,\R)$ and the sequences 
$\{L^{1/2} z_m\}_m$ and $\left\{\frac 1{\sqrt{m}} z_m\right\}_m$ are uniformly bounded 
in $W^{1,2}(0,{T};L^{2}(\Om,\R^N))$. Indeed, multiplying (\ref{microPr3Dis}) 
by the term $(\Sig_{n,m}-\Sig_{n-1,m})/{h}$ and integrating over $\Om$ we obtain
\begin{align}
\label{est_n1}
\left( \frac{z_m^n-z_m^{n-1}}{h}, \frac{\Sig_{n,m}-\Sig_{n-1,m}}{h}\right)_\Om&=\frac 1h(\partial f(\Sig_{n,m}),\Sig_{n,m}-\Sig_{n-1,m})_{\Om}\non\\&\geq \frac 1h (F(\Sig_{n,m})-F(\Sig_{n-1,m})),
\end{align}
where $F:L^2(\Om)\to\bar\R$ is a convex functional defined by 
$$F (v)=\begin{cases}\int_{\Om}f(x,v(x))dx,& f(\cdot,v(\cdot))\in L^1(\Om), \\+\infty,& {\rm otherwise}.\end{cases}$$
In the following we use the notation $M:= -B^T{\rm sym}\C Q_2{\rm sym}B$, $\partial\Phi:= C_1B^T{\rm Curl Curl}B$ and $\hat{z}_m^n:=B^T(\si_m^n+\tau_m^n)$ are used. 
Since $M$ and $\partial\Phi$ are linear positive semi-definite, symmetric  operators we can rewrite the left side of (\ref{est_n1}) as follows
\begin{align}
\label{est_n2}
&\left( \frac{z_m^n-z_m^{n-1}}{h}, \frac{\Sig_{n,m}-\Sig_{n-1,m}}{h}\right)_\Omega
=-\left( \frac{z_m^n-z_m^{n-1}}{h}, \left(M+L+\frac 1m\right) \frac{z_m^n-z_m^{n-1}}{h}\right)_\Om\non\\ 
-&\left( \frac{z_m^n-z_m^{n-1}}{h}, \frac{ \partial\Phi(z_m^n)-\partial\Phi(z_m^{n-1})}{h}\right)_\Om+
\left( \frac{z_m^n-z_m^{n-1}}{h}, \frac{\hat{z}_m^n-\hat{z}_m^{n-1}}{h}\right)_\Om\\
=&-\left\|M^{1/2} \frac{z_m^n-z_m^{n-1}}{h}\right\|^2_2-\left\|L^{1/2} \frac{z_m^n-z_m^{n-1}}{h}\right\|^2_2-
\left\|\frac 1{\sqrt{m}} \frac{z_m^n-z_m^{n-1}}{h}\right\|^2_2\non\\&-
C_1\left\|\Curl B \frac{z_m^n-z_m^{n-1}}{h}\right\|^2_2+
\left( \frac{z_m^n-z_m^{n-1}}{h}, \frac{\hat{z}_m^n-\hat{z}_m^{n-1}}{h}\right)_\Om\non.
\end{align}
Now we combine (\ref{est_n1}) and (\ref{est_n2}), multiply  the obtained 
relation by $h$ and sum it up for $n=1,...,l$ 
and any fixed $l\in\{1,...,2^m\}$. We obtain 
\begin{align}
\label{est_n2'}
h&\sum_{n=1}^l \Big(\left\|M^{1/2} \frac{z_m^n-z_m^{n-1}}{h}\right\|^2_2+
\left\|L^{1/2} \frac{z_m^n-z_m^{n-1}}{h}\right\|^2_2+\left\|\frac 1{\sqrt{m}} \frac{z_m^n-z_m^{n-1}}{h}\right\|^2_2\non\\&+C_1\left\|\Curl B \frac{z_m^n-z_m^{n-1}}{h}\right\|^2_2\Big)+\int_{\Om}f(x,\Sig_{l,m}(x)) dx\non\\
&\leq F(\Sig_{0,m})+h\sum_{n=1}^l \left( \frac{z_m^n-z_m^{n-1}}{h}, \frac{\hat{z}_m^n-\hat{z}_m^{n-1}}{h}\right)_\Om,
\end{align}
 which implies the estimate
 \begin{align}
 \label{est_n3}
&\|M^{1/2} z_{mt}\|^2_{2,\Om_{t}}+\|L^{1/2} z_{mt}\|^2_{2,\Om_{t}}+\left\|\frac1{\sqrt{m}} z_{mt}\right\|^2_{2,\Om_{t}}+C_1\|\Curl B z_{mt}\|^2_{2,\Om_{t}}\non\\&+\|f(\bar{\Sigma}^{\rm lin}_{m}(t))\|_{1,\Om}\leq F(\Sig(0))+\|z_{mt}\|_{q^*,\Om_{t}} \|\hat{z}_{mt}\|_{q,\Om_{t}}. 
 \end{align}
 Since by (\ref{aprioriEstim3}) the right side of (\ref{est_n3}) is bounded we obtain
\begin{align}
&\left\{f(\Sigbar_{m})\right\}_m\ {\rm is}\  {\rm uniformly}\  {\rm bounded}\  {\rm in}
 \ L^\infty(0,{T_e};L^{1}(\Om,{\R})),\label{aprioriEstim2'}\\
&\{L^{1/2} z_m\}_m \ {\rm is}\  {\rm uniformly}\  {\rm bounded}\  {\rm in}
 \ W^{1,2}(0,{T_e};L^{2}(\Om, {\R}^N)),\label{aprioriEstim3'}\\[1ex] 
& \left\{\frac1{\sqrt{m}} z_m\right\}_m \ {\rm is}\ 
  {\rm uniformly}\  {\rm bounded}\  {\rm in}
 \ W^{1,2}(0,{T_e};L^{2}(\Om,{\R}^N)),\label{aprioriEstim4'}\\
 & \left\{M^{1/2}z_m\right\}_m \ {\rm is}\ 
  {\rm uniformly}\  {\rm bounded}\  {\rm in}
 \ W^{1,2}(0,{T_e};L^{2}(\Om, {\R}^N)),\label{aprioriEstim5'}\\
 & \left\{\Curl Bz_m\right\}_m \ {\rm is}\ 
  {\rm uniformly}\  {\rm bounded}\  {\rm in}
 \ W^{1,2}(0,{T_e};L^{2}(\Om, {M}^3)).\label{aprioriEstim6'}
 \end{align}
Furthermore, due to the self-controlling property (\ref{SelfControlling}) and 
(\ref{aprioriEstim3'}), (\ref{aprioriEstim4cc}) we obtain that
\begin{eqnarray}
&&\{B(\partial_tz_m)\}_m \ {\rm is}\  {\rm uniformly}\  {\rm bounded}\  {\rm in}
 \ L^2(0,{T_e};Z^{2}_{\Curl}(\Omega,{\cal M}^3)).
\label{aprioriEstim3new3}
 \end{eqnarray}
 Moreover, (\ref{RotheAffineinterpolant}) and (\ref{RotheConstantinterpolant}) yield 
$\{Bz_m(x, t), B\bar z_m(x,t)\}_m\in\mathfrak{sl}(3)$ for a.e. $(x, t)\in\Omega_{T_e}$.
\vspace{1ex}\\
{\bf Additional regularity of discrete solutions.}  
In order to get the additional 
a'priori estimates, we extend the function $b$ to $t<0$ by setting
$b(t)=b(0)$. 
The extended function $b$ is in the space
$W^{1,p}(-2h,T_e;W^{-1,p}(\Omega,{\mathbb R}^3))$. Then, we set
 $b^0_m=b^{-1}_m:=b(0)$. 
Let us further set $$z_m^{-1}:=z_m^0-h {\cal G}(\Sigma^{\rm lin}_{0,m}).$$
The assumption (\ref{Assumption1}) implies that $z_m^{-1}=0$.
Next, we define functions
$(u_m^{-1}, \sigma_m^{-1})$ and $(u_m^0, \sigma_m^0)$ as solutions of the 
linear elasticity problem (\ref{elast1}) - (\ref{elast3})
to the data $\hat b=b^{-1}_m$, $\hat\ve_p=0$ 
and $\hat b=b^0_m$, $\hat\ve_p=0$,
respectively. Obviously, 
the following estimate holds
\begin{eqnarray}\label{IneqPrepare}
\left\{
\left\|\frac{u_m^0-u_m^{-1}}h\right\|_{2},
\left\|\frac{\sigma_m^0-\sigma_m^{-1}}h\right\|_{2}\right\}\le C,
\end{eqnarray}
where $C$ is some positive constant independent of $m$.
Taking now the incremental ratio of (\ref{CurlPr3Dis}) for
$n=1,...,2^m$, we obtain\footnote{For sake of
simplicity we use the following notation 
$\rt\phi^n_m:=(\phi^n_m-\phi^{n-1}_m)/h$, where 
 $\phi^{0}_m,\phi^{1}_m,...,\phi^{m}_m$ is any family of functions.}
\[\rt z^n_m-\rt z^{n-1}_m={\cal G}(\Sigma ^{\rm lin}_{n,m})-
{\cal G}(\Sigma ^{\rm lin}_{(n-1),m}).\]
Let us now multiply the last identity by 
$-(\Sigma^{\rm lin}_{n,m}- \Sigma^{\rm lin}_{(n-1),m})/h$.
Then using the monotonicity of ${\cal G}$ we obtain that
\[\frac{1}{m} \Big(\rt z^n_m-\rt z^{n-1}_m, \rt z^n_m\Big)_\Omega
+\big(\rt z^n_m-\rt z^{n-1}_m, L\rt z^n_m\big)_\Omega\]\[
+\big(\rt z^n_m-\rt z^{n-1}_m, B^T\Curl\Curl B(\rt z^n_m)\big)_\Omega
\]\[\le\big(\rt z^n_m-\rt z^{n-1}_m,C_1B^T\rt\sigma^{n}_m\big)_\Omega
+\big(\rt z^n_m-\rt z^{n-1}_m,B^T \rt\hat\sigma^{n}_m\big)_\Omega.
\]
 With 
(\ref{CurlPr1Dis}) and (\ref{CurlPr2Dis}) the previous inequality can be
rewritten as follows
\[\frac{1}{m} \Big(\rt z^n_m-\rt z^{n-1}_m, \rt z^n_m\Big)_\Omega
+\big(\rt z^n_m-\rt z^{n-1}_m, L\rt z^n_m\big)_\Omega\]\[
+\big(\rt z^n_m-\rt z^{n-1}_m, B^T\Curl\Curl B(\rt z^n_m)\big)_\Omega
+\big(\rt \sigma^n_m-\rt \sigma^{n-1}_m,{\mathbb C}^{-1}\rt\sigma^{n}_m\big)_\Omega
\]\[\le\big(\rt \hat\sigma^n_m-\rt\hat\sigma^{n-1}_m,\bar{\mathbb C}\rt\sigma^{n}_m\big)_\Omega
+\big(\rt z^n_m-\rt z^{n-1}_m, B^T\rt\hat\sigma^{n}_m\big)_\Omega.
\]
As in the proof of (\ref{AprioriEstimHelp1}),
 multiplying  the last inequality by $h$ and summing with respect to $n$
 from 1 to $l$ 
for any fixed $l\in[1,2^m]$ we get the estimate
\begin{eqnarray}\label{IneqExstSols}
&&\frac{h}{m}\|\rt z^l_m\|^2_{2}+h\|L^{1/2}\rt z^l_m\|^2_{2}+
h\|{\mathbb B}^{1/2}\rt\sigma^l_m\|^2_{2}+h\|\Curl B\rt z^l_m\|^2_{2}
\le 2hC^{(0)}\non\\[1ex]
&&
+ 2h\sum^l_{n=1}\left(B^T\rt\hat\sigma^n_m,\rt z^n_m-\rt z^{n-1}_m\right)_\Omega
+ 2h\sum^l_{n=1}\big(\rt \hat\sigma^n_m-\rt\hat\sigma^{n-1}_m,\bar{\mathbb C}\rt\sigma^{n}_m\big)_\Omega,
\end{eqnarray}
where now $C^{(0)}$ denotes
\[2C^{(0)}:=\|{\mathbb B}^{1/2}\rt \sigma^0_m\|^2_{2}.\]
We note that \eq{IneqPrepare} yields the uniform
boundness of $C^{(0)}$ with respect to $m$.
Summing now (\ref{IneqExstSols}) for $l=1,...,2^m$ we derive the inequality
\begin{eqnarray}\label{aprioriEstim21}
\frac{1}{m}\left\|\partial_tz_m\right\|^2_{2,\Omega_{T_e}}+
\left\|L^{1/2}\left(\partial_tz_m\right)\right\|^2_{2,\Omega_{T_e}}+
\left\|\Curl B\left(\partial_tz_m\right)\right\|^2_{2,\Omega_{T_e}}\\
+C\left\|\partial_t\sigma_m\right\|^2_{2,\Omega_{T_e}}\le 
C\|\partial_t\hat\sigma_m\|_{2,\Omega_{T_e}}
(\|\partial_t\sigma_m\|_{2,\Omega_{T_e}}+\|B(\partial_tz_m)\|_{2,\Omega_{T_e}}).\non
\end{eqnarray}
Using the self-controlling property (\ref{SelfControlling}) and Young's
inequality with $\epsilon>0$ we obtain that
\begin{eqnarray}\label{aprioriEstim22}
\frac{1}{m}\left\|\partial_tz_m\right\|^2_{2,\Omega_{T_e}}+
\left\|L^{1/2}\left(\partial_tz_m\right)\right\|^2_{2,\Omega_{T_e}}+
\left\|\Curl B\left(\partial_tz_m\right)\right\|^2_{2,\Omega_{T_e}}+
C_\epsilon\left\|\partial_t\sigma_m\right\|^2_{2,\Omega_{T_e}}\\
\le 
C\|\partial_t\hat\sigma_m\|^2_{2,\Omega_{T_e}} +\|B(\partial_tz_m)\|_{2,\Omega_{T_e}}).\non
\end{eqnarray}
Since $\hat\sigma_m$ is uniformly bounded in 
$W^{1,q}(\Omega_{T_e}, {\cal S}^3)$ and $B(\partial_tz_m)$ is uniformly bounded in
 $L^2(0,{T_e};L^{2}(\Omega,{\cal M}^3))$, 
estimates (\ref{aprioriEstim21}) and (\ref{aprioriEstim22}) imply 
\begin{eqnarray}
&&\{L^{1/2}(\partial_t z_m)\}_m \ {\rm is}\  {\rm uniformly}\  {\rm bounded}\  {\rm in}
 \ L^2(0,{T_e};L^{2}(\Omega,{\mathbb R}^N)),\label{aprioriEstim7}\\[1ex]
&&\{\partial_t\sigma_m\}_m\ {\rm is}\  {\rm uniformly}\  {\rm bounded}\  {\rm in}
 \ L^2(0,{T_e};L^{2}(\Omega,{\mathbb R}^N)),\label{aprioriEstim7a}\\[1ex]
&&\{\Curl B(\partial_t z_m)\}_m\ {\rm is}\  {\rm uniformly}\  {\rm bounded}\  {\rm in}
 \ L^2(0,{T_e};L^{2}(\Omega,{\cal M}^3)),\label{aprioriEstim7aa}\\[1ex]
&&\left\{\frac1{\sqrt{m}} \partial_t z_m\right\}_m\ {\rm is}\  {\rm uniformly}\  {\rm bounded}\  {\rm in}
 \ L^2(0,{T_e};L^{2}(\Omega,{\mathbb R}^N))\label{aprioriEstim7b},\label{aprioriEstim8}\\
&&\{Bz_m\}_m \ {\rm is}\  {\rm uniformly}\  {\rm bounded}\  {\rm in}
 \ H^1(0,{T_e};L^{2}_{\Curl}(\Omega,{\cal M}^3)). \label{aprioriEstim9}
\end{eqnarray}
{\bf Existence of solutions.} At the expense of extracting a subsequence, by estimates 
(\ref{aprioriEstim3}) - (\ref{aprioriEstim6}), (\ref{aprioriEstim2'}) - (\ref{aprioriEstim3new3}) and
(\ref{aprioriEstim7}) - (\ref{aprioriEstim9}) 
 we obtain that the sequences
in (\ref{aprioriEstim3}) - (\ref{aprioriEstim3new3}),  
(\ref{aprioriEstim7}) - (\ref{aprioriEstim9}) converge with respect to
 weak and weak-star topologies in corresponding spaces, respectively.
Next, we claim that weak limits of $\{\bar z_m\}_m$ and $\{z_m\}_m$ coincide.
Indeed, using (\ref{aprioriEstim3}) this can be shown as follows
\[\|z_m-\bar z_m\|^{q^*}_{q^*,\Omega_{T_e}}=\sum_{n=1}^m
\int_{(n-1)h}^{nh}\left\|(z^{n}_m-z^{n-1}_m)\frac{t-nh}{h}\right\|^{q^*}_{q^*}dt\]
\[=\frac{h^{{2}+1}}{{2}+1}\sum_{n=1}^m
\left\|\frac{z^{n}_m-z^{n-1}_m}{h}\right\|^{q^*}_{q^*}=\frac{h^{2}}{{2}+1}
\left\|\frac{dz_m}{dt}\right\|^{q^*}_{q^*,\Omega_{T_e}},\]
which implies that  $\bar z_m-z_m$ converges strongly to 0  in 
$ L^{q^*}(\Omega_{T_e},{\mathbb R}^N)$. The proof of the fact
 that the difference
$\bar \sigma_m-\sigma_m$ converges weakly to 0 in 
$L^{2}(\Omega_{T_e},{\cal S}^3)$ can be performed as in \cite[p. 210]{Roubi05}.
For the reader's convenience we reproduce here the reasoning from there.
Let us choose some appropriate number $d\in{\mathbb N}$ and then fix
any integer $n_0\in[1,2^d]$. Let $h_0=T_e/2^{n_0}$. Consider functions
 $I_{[h_0(n_0-1), h_0n_0]}v$ with  
$v\in L^{2}(\Omega,{\cal S}^3)$, where $I_K$ denotes the indicator function
of a set $K$. We note that, according to \cite[Proposition 1.36]{Roubi05},
 the linear combinations of all such functions
are dense in $L^{2}(\Omega_{T_e},{\cal S}^3)$. Then for any $h\le h_0$\footnote{We recall that $h$ is chosen to be equal to $T_e/2^{m}$ 
for some $m\in{\mathbb N}$.}
\[\left(\sigma_m-\bar \sigma_m,I_{[h_0(n_0-1), h_0n_0]}v\right)_{\Omega_{T_e}}
=\int_{h_0(n_0-1)}^{h_0n_0}\left(\sigma_m(t)-\bar \sigma_m(t),v\right)_{\Omega}dt\]
\[=\sum_{n=h_0(n_0-1)/h+1}^{h_0n_0/h}\int_{(n-1)h}^{nh}
\left((\sigma^n_m-\sigma^{n-1}_m)\frac{t-nh}{h},v\right)_{\Omega}dt\]
\[=-\frac{h}2\left(\sigma_m^{h_0n_0/h}-\sigma^{h_0(n_0-1)/h}_m,v\right)_{\Omega}=
-\frac{h}2\left(\bar\sigma_m(h_0n_0)-\bar\sigma_m(h_0(n_0-1)),v\right)_{\Omega}.\]
Employing (\ref{aprioriEstim3a}) we get that 
$\bar \sigma_m-\sigma_m$ converges weakly to 0 in 
$L^{2}(\Omega_{T_e},{\cal S}^3)$. Next, by (\ref{aprioriEstim3b}) the
sequence $\{z_m/{m}\}_m$ converges strongly to 0 in 
$L^{2}(\Omega_{T_e},{\R}^N)$. Summarizing all observations made above
we may conclude that the limit functions denoted by $\tilde v, \tilde\sigma, z$ and
$\Sigma^{\rm lin}$ have the following properties
\[(\tilde v, \tilde\sigma)\in H^{1}(0,T_e; H^{1}_0(\Omega, {\mathbb R}^3) \times
 L^{2} ({\Omega}, {\cal S}^3)),\] 
\[
\Sigma^{\rm lin}\in L^{q}(\Omega_{T_e}, {\R}^N),, \ \ z\in W^{1,q^*}(0,T_e;L^{q^*}(\Omega, \R^N)),\]
\[ Bz\in   
H^{1}(0,T_e;L^{2}_{\Curl}(\Omega,{\cal M}^3))\cap 
L^{2}(0,{T_e};Z^{2}_{\Curl}(\Omega,{\cal M}^3)).\]
Moreover,  
$Bz(x, t)\in\mathfrak{sl}(3)$ holds for a.e. $(x, t)\in\Omega_{T_e}$.
Before passing to the weak limit, we note that the 
Rothe approximation functions
satisfy the equations
\begin{eqnarray}
- \di_x\bar \sigma_m(x,t) &=&  0, \label{Curl1Dis}
\\[1ex]
\sigma_m(x,t) &=& {\mathbb C}(\sym(\na_xu_m(x,t)-Bz_m(x,t)))  \label{Curl2Dis}\\[1ex]
&&\hspace{7ex}+({\mathbb C}[x] -\hat{\mathbb C})(\hat{\mathbb C})^{-1}\hat\sigma_m(x), \non
\\[1ex] 
\label{Curl3Dis} \partial_t z_m(x,t) & \in & 
g \big(\bar\Sigma^{\rm lin}_{m}(x,t)\big),\label{micro3Dis} 
\end{eqnarray}
together with the initial and  boundary conditions   
\begin{eqnarray} 
z_m(x,0) &=& 0, \ \ x \in \Omega,\label{Curl4Dis}\\
 Bz_m(x,t)\times n (x)&=&0,\hspace{7ex} x \in \partial\Omega, 
\label{Curl5Dis}\\
u_m(x,t) &=& 0, \hspace{7ex} x \in \partial \Omega\,.\label{Curl6Dis}
\end{eqnarray}
Passing to the weak limit in (\ref{Curl1Dis}), (\ref{Curl2Dis}) and (\ref{Curl4Dis}) - (\ref{Curl6Dis})
we obtain that the limit functions $\tilde v, \tilde\sigma, z$ satisfy equations 
(\ref{CurlPr1A}), (\ref{CurlPr2A}) and (\ref{CurlPr4A}) - (\ref{CurlPr6A}). 
To show that the limit functions
satisfy also (\ref{CurlPr3A}) we proceed as follows:\\
As above, the system (\ref{Curl1Dis}) - (\ref{Curl6Dis}) can be rewritten
as
\begin{eqnarray}
\label{IneqWeakSolHelp2}
&&\int_0^{T_e}\int_\Omega\left(g^{-1}(\partial_tz_m(x,t))\cdot\partial_tz_m(x,t))\right)dxdt=
-\left(\frac{d\sigma_m}{dt},
{\mathbb C}^{-1}\bar\sigma_m\right)_{\Omega_{T_e}}\non\\
&&-
 \left(\frac{dz_m}{dt},L\bar z_m\right)_{\Omega_{T_e}}
-\frac1{m}\left(\frac{dz_m}{dt},\bar z_m\right)_{\Omega_{T_e}}\\
&&- C_1\left(\frac{dz_m}{dt},B^T\Curl\Curl B\bar z_m\right)_{\Omega_{T_e}}+
\left(B^T\hat\sigma_m,\partial_tz_m\right)_{\Omega_{T_e}}
+\left(\bar{\mathbb C}\bar\sigma_m, \partial_t\hat\sigma_m\right)_{\Omega_{T_e}}.\non
\end{eqnarray}
Due to (\ref{aprioriEstim7}) - (\ref{aprioriEstim9})
  we can pass to
the weak limit inferior in (\ref{IneqWeakSolHelp2}) to get the
following inequality
\begin{eqnarray}
\label{IneqWeakSolHelp3}
&&\limsup_{m\to\infty}\int_0^{T_e}\int_\Omega\left(g^{-1}(\partial_tz_m(x,t))\cdot\partial_tz_m(x,t))\right)dxdt\\
&&\le
\left(\partial_tz, B^T(\tilde\sigma+\hat\sigma)-L z-B^T\Curl\Curl Bz\right)_{\Omega_{T_e}}.\non
\end{eqnarray}
Let ${\cal G}$ denote the canonical
extension of $g$. Then (\ref{IneqWeakSolHelp3}) reads as follows
\begin{eqnarray}
\label{IneqWeakSolHelp4}
\limsup_{m\to\infty}\left({\cal G}^{-1}(\partial_tz_m),\partial_tz_m\right)_{\Omega_{T_e}}
\le
\left(\partial_tz,B^T(\tilde\sigma+\hat\sigma)-Lz-B^T\Curl\Curl Bz\right)_{\Omega_{T_e}}.
\end{eqnarray}
Since ${\cal G}^{-1}$ is pseudo-monotone, inequality (\ref{IneqWeakSolHelp4}) yields that
for a.e. $(x, t)\in \Omega_{T_e}$
\[\partial_tz(x,s)\in g(B^T(\tilde\sigma(x, t)+\hat\sigma(x, t))-L z(x, t)-B^T\Curl\Curl Bz(x, t)).\]
Therefore, we conclude that 
the limit functions $\tilde v, \tilde\sigma, z$ and $\Sigma^{\rm lin}$ satisfy equations
(\ref{CurlPr1A}) - (\ref{CurlPr6A})  and the existence of strong solutions is herewith established.
 
  This completes the proof
of Theorem~\ref{existMain}.
\end{proof}



\section{Uniqueness of strong solutions}\label{Existence}

In this section we present the uniqueness result for (\ref{CurlPr1}) - (\ref{CurlPr6}) with a function $g$ satisfying the self-controlling condition (\ref{SelfControlling}). A function $g$ having the property (\ref{SelfControlling}) is automatically single-valued. Having noticed this we obtain the following uniqueness result.
\begin{Theor}\label{uniqueMain} 
Let all assumptions of Theorem~\ref{existMain} be satisfied.
Then the solution $(u,\sigma,z)$
of the initial boundary value problem (\ref{CurlPr1}) - (\ref{CurlPr6}) is unique.
\end{Theor} 
\begin{proof} Let $(u_1,\sigma_1,z_1)$ and $(u_2,\sigma_2,z_2)$ be two solutions
of the initial boundary value problem (\ref{CurlPr1}) - (\ref{CurlPr6}). Next, we argue similarly as in the proof of Theorem~\ref{existMain}:
The monotonicity of $g$ implies that
\[\big( \partial_t z_1(x,t) - \partial_t z_2(x,t), \Sigma^{\rm lin}_1(x,t) - \Sigma^{\rm lin}_2(x,t)\big)\ge0\]
for a.e. $(x,t)\in\Omega_{T_e}$, where
\[ \Sigma^{\rm lin}_i=B^T\sigma_i-Lz_i- B^T \Curl\Curl (Bz_i), \ \ i=1,2.\]
Integrating the last inequality over $\Omega_{t}$ with $t\in(0, T_e)$ 
and using the equations (\ref{CurlPr1}) and
(\ref{CurlPr2}) we get the following estimate for the difference of the solutions 
(here ${\mathbb B}:={\mathbb C}^{-1}$)
\begin{eqnarray}
\label{IneqUnique1}
0&\ge& \|\Curl p_1(t)-\Curl p_2(t)\|_2^2+\|{\mathbb B}^{1/2}(\sigma_1(t)- \sigma_2(t))\|_2^2\non\\
&&\hspace{1ex}+\|{L}^{1/2}(z_1(t)- z_2(t))\|_2^2,
\end{eqnarray}
which holds for a.e. $t\in(0, T_e)$. The estimate (\ref{IneqUnique1}) together with the condition
(\ref{IsotropStuff}) imply that $\Sigma^{\rm lin}_1(x,t)=\Sigma^{\rm lin}_2(x,t)$
for a.e. $(x,t)\in\Omega_{T_e}$. Since the function $g$ is single-valued and 
$(u_i,\sigma_i,z_i)$, $i=1,2$, are the solutions of the problem (\ref{CurlPr1}) - (\ref{CurlPr6}),
the last identity yields that $p_1(x,t)=p_2(x,t)$ for a.e. $(x,t)\in\Omega_{T_e}$. 
This completes the proof of the theorem.
\end{proof}
\ \vspace{1ex}\\
{\bf Acknowledgement.}

The authors thank Krzysztof Chelminski and Patrizio Neff for the critical reading of the manuscript 
and valuable comments which resulted in the improvements of the proofs and of the presentation 
of the obtained results.

\bibliographystyle{plain} 
{\footnotesize
\bibliography{literatur1,literaturliste}
}

\end{document}